\documentclass{amsart}
\usepackage{amssymb}

\newtheorem*{unnumberedtheorem}{Theorem}

\newtheorem*{proposition}{Proposition}
\newtheorem{lemma}{Lemma}
\theoremstyle{definition}

\DeclareMathOperator{\init}{in}
\DeclareMathOperator{\inv}{inv}
\DeclareMathOperator{\ord}{ord}

\DeclareMathOperator{\coeff}{coeff}

\DeclareMathOperator{\Spec}{Spec}

\newcommand{\N}{\mathbb{N}}
\newcommand{\eps}{\varepsilon}
\newcommand{\w}{w}
\newcommand{\wt}{\widetilde}
\newcommand{\ol}{\overline}
\newcommand{\x}{x}
\newcommand{\y}{y}
\newcommand{\Jo}{K}
\newcommand{\rs}{\sum_{j\notin\A}s_j}
\newcommand{\tJo}{K_1}
\newcommand{\tJop}{K_1'}
\newcommand{\oc}{\frac{o}{c!}}

\newcommand{\kk}{\mathbb{K}}
\newcommand{\cic}{\frac{c-i}{c!}}
\newcommand{\qc}{\frac{\q}{c!}}

\newcommand{\z}{u}
\newcommand{\zp}{u'}
\newcommand{\zpp}{\wt u}
\newcommand{\coeffJ}{\coeff_{V}(J)}

\newcommand{\q}{{p^e}}
\newcommand{\A}{T}
\newcommand{\g}{q}
\newcommand{\gp}{g}
\newcommand{\m}{w}
\newcommand{\G}{H}
\newcommand{\dd}{\ord I}
\newcommand{\ddp}{\ord I_1'}
\newcommand{\V}{U}
\newcommand{\Vp}{U'}
\newcommand{\Vpp}{\wt U}
\newcommand{\Kpp}{\wt K}
\newcommand{\dds}{\Big(\dd+\rs\Big)}
\newcommand{\To}{T\setminus\{1\}}
\newcommand{\mm}{m_R}
\newcommand{\mmp}{m_{R'}}

\newcommand{\wF}{\wt F}
\newcommand{\OO}{\mathcal{O}}
\def\u{v}

\font\bf=ptmb at 10pt

\font\tss=ptmr at 8pt

\def\ltextindent#1{\hbox to \hangindent{#1\hss}\ignorespaces}
\long\def\ignore#1\recognize{}
\def\no{\noindent}
\def\big{\bigskip}
\def\med{\medskip}
\def\hs{\hskip}

\def\gb{\goodbreak}
\def\cl{\centerline}
\def\ds{\displaystyle}

\def\ol{\overline}

\def\wt{\widetilde}

\def\sm{\setminus}

\def\{{\lbrace}
\def\}{\rbrace}

\def\isom{\cong}
\def\map{\rightarrow}

\def\inv{^{-1}}
\def\6{\partial}
\def\Spec{{\rm Spec}}
\def\Proj{{\rm Proj}}

\def\m{m}

\def\K{{\Bbb K}}
\def\KK{{K}}

\def\<{<\!}
\def\>{\!>}

\def\inin{{\rm in}}

\def\TT{{} i\!i_\TT}

\def\tt{t\!t}
\def\TT{T}

\def\1{\bf 1}

\def\inv{{\rm inv}}
\def\lex{{\rm lex}}

\def\rrho{,\rho}
\def\rrho{}

\newcommand{\hh}{}
\newcommand{\QQ}{Q}
\newcommand{\congruent}{\equiv}

\newcommand{\st}{}

\begin{document}


\title[Characterizing the Increase of the Residual Order]{Characterizing the Increase of the Residual Order under Blowup in Positive Characteristic
}

\author{Herwig Hauser, Stefan Perlega}
\maketitle


\let\thefootnote\relax\footnote{
{\tss \hs -.41cm
Communicated by S.~Mukai. Received April 21, 2017. Revised May 7, 2018; May 17, 2018.\par\no
MSC-2000: 14B05, 14E15, 12D10. \par\no
The authors are grateful to an anonymous referee for a very careful reading and many valuable suggestions. Supported by projects P-25652 and P-31338 of the Austrian Science Fund FWF. \med

\no Faculty of Mathematics, University of Vienna, Austria.\par
\no herwig.hauser@univie.ac.at, stefan.perlega@univie.ac.at
}}


\begin{abstract} In contrast to the characteristic zero situation, the residual order of an ideal may increase in positive characteristic under permissible blowups at points of the exceptional divisor where the order of the ideal has remained constant. The specific situations where this happens are described explicitly.\end{abstract}



\section{Introduction}

\no To prove embedded resolution of singularities in characteristic zero for a reduced subscheme $X$ of a regular ambient scheme $W$ equipped with a normal crossings divisor $D$ one typically associates to every point $a$ of $X$ a local invariant $\inv_aX$ measuring the complexity of the singularity of $X$ at $a$ and the position of $X$ with respect to $D$. The invariant consists of a string of non-negative integers, is upper semicontinuous and decreases lexicographically when $X$ is blown up along the center $Z$ defined as the locus of points where $\inv_aX$ attains its maximal value. This is done in a way so that $Z$ is regular and has normal crossings with $D$. As the invariant varies in the well ordered set $(\N^N,\lex)$ and its minimal value corresponds to a regular point $a$ at which $X$ has normal crossings with $D$, the resolution of $X$ is obtained by induction \cite{Hironaka_Annals, Villamayor_Constructiveness, Villamayor_Patching, BM_Canonical_Desing, EH, Cutkosky_Book, Wlodarczyk, Kollar_Book}.

For the first component of $\inv_aX$ the simplest choice is the order $\ord_aJ$ of the defining ideal $J$ of $X$ in $W$. Blowing up a regular center contained in the associated equimultiple locus of $J$, the order does not increase, $\ord_{a'}J'\leq\ord_aJ$, for all points $a'$ in the weak transform $X'$ of $X$ above $a$. At a point $a'$ where the order remains constant, the second component of $\inv_aX$ comes into play. Leaving aside transversality issues of $Z$ with $D$, it is (usually) defined as the order of the coefficient ideal $K$ of $J$ at $a$ with respect to a hypersurface of maximal contact $V$, less the exceptional multiplicity of $K$. This numeral \hh
does not depend on the choice of the hypersurface and is again upper semicontinuous along the strata defined by the order of $J$. It is thus well suited to form the second component of $\inv_aX$. Blowing up a regular center $Z$ inside the top loci of the order and of the \hh 
order of the coefficient ideal, the second component does not increase whenever the first remains constant. From this point on, the argument is repeated until, by exhaustion of dimensions, a decrease of the invariant under blowup is established.

This approach to resolution has several drawbacks in positive characteristic: First, hypersurfaces of maximal contact no longer exist; a possible substitute are hypersurfaces of {\it weak} maximal contact as introduced in \cite{EH, Ha_BAMS_1, Ha_BAMS_2}. These are defined as regular hypersurfaces maximizing the order of the coefficient ideal (in characteristic zero, the maximum can be realized by a hypersurface of maximal contact.) \hh This maximum will be called the \emph{residual order} of $J$ at $a$ (the name residual order was introduced by Hironaka for the situation in positive characteristic in \cite{Hironaka_CMI}). Secondly, the \hh residual order is no longer upper semicontinuous, so its top locus need not be closed; extra care has to be taken. Finally, even if centers are chosen appropriately, the residual order may still go up under blowup at points where the order of the ideal $J$ has remained constant. This increase is also known as the ``kangaroo phenomenon''. It destroys the induction argument.

In view of these difficulties, two approaches are plausible: Either to reject the residual order as a valuable resolution invariant in positive characteristic and to search for new invariants. This option has been undertaken with a certain success by several authors \cite{Hironaka_Bowdoin, Cossart_Thesis, Cossart_Contact_Maximal, Villamayor_Hypersurface, Kawanoue_IF_1, Kawanoue_Matsuki}. Or, to try to understand better the circumstances where the residual order behaves badly in order to develop an exit strategy for the obstructions. This is the proposal we wish to pursue in the present paper.

In this spirit, the situations where an increase of the residual order occurs under blowup with permissible choices of centers will be investigated in detail. It turns out that in order to produce an increase, the defining equations of $X$ in $W$ must satisfy quite restrictive conditions: The (weighted) initial forms of minimal order of the elements of $J$ have a unique form (up to constant factors and coordinate changes), and are actually powers of purely inseparable polynomials. Their logarithmic Hasse derivatives have a specific shape, and the exceptional multiplicities of the coefficient ideal $K$ of $J$ satisfy an explicit arithmetic inequality. These three conditions are satisfied simultaneously only in very special cases.

Along the proof of these facts, we extend Moh's bound on the possible increase of the residual order to non-hypersurfaces and not necessarily purely inseparable power series \cite{Moh}. The upshot of the results is as follows (see section 3 for the precise statements):\med


\begin{unnumberedtheorem} Let $J$ be an ideal of $W$ of order $c$ at $a$, with coefficient ideal $K$ of order $o$ with respect to a hypersurface of weak maximal contact. Assume that the residual order of $J$ with respect to a given normal crossings divisor $D$ increases under permissible blowup at a point $a'$ where $c$ has remained constant. Then $c$ is a multiple $m\cdot p^e$ of a power of the characteristic $p$, $o$ is a multiple $w\cdot c!$ of $c!$, the weighted initial form with respect to $w$ of elements $f$ of $J$ of minimal weighted order  is a power $\inin_w(f)=(z^c+F(x))^m$, with $F$ a homogeneous polynomial of degree $w\cdot p^e$ in variables $x_1,\ldots,x_n$, not a $p^e$-th power, and with a specific \hh shape  $x_k^{p^\ell}\6_{x_k^{p^\ell}}F$ of the logarithmic Hasse derivative. Here, $\ell<e$ is maximal so that $F$ is a $p^\ell$-th power. 

Choosing $x_i$ subordinate to $D$ and factorizing $F$ maximally into $F=x^r\cdot G$, the residues modulo $p^{\ell+1}$ of the exceptional exponents $r_i$ satisfy $\sum \ol {r_i} \leq (b-1)\cdot p^{\ell+1}$, where the sum ranges over those exceptional components which are lost when passing from $a$ to $a'$, and where $b$ is the number of \hh $r_i\not\congruent 0$ modulo $p^{\ell+1}$ among them.

 In the above situation, the residual order of $J$ increases at most by $c!\over p$.
\end{unnumberedtheorem}


Various other notable approaches to the resolution problem in positive characteristic can be found in \cite{Abhyankar_56, Giraud, Hironaka_Bowdoin, Cossart_Thesis, Cossart_Contact_Maximal, Moh_Newton_Polygon, Cutkosky_Book, Cutkosky_Skeleton, Ha_Power_Series, HW, Villamayor_Hypersurface, BV_Monoidal, Kawanoue_IF_1, Kawanoue_Matsuki}. \med


{\it Acknowledgments.} We are indebted to the following mathematicians for many valuable discussions and suggestions: S.~Abhyankar, H.~Hironaka, G.~M\"uller, J.~Schicho, A.~Quir\'os,  S.~Encinas, O.~Villamayor, A.~Bravo, D.~Cutkosky, J.-J.~Risler, V.~Cossart, J.~W\l odarczyk, H.~Kawanoue, K.~Matsuki, D.~Panazzolo, M.~Spivakovsky, F.~Cano, R.~Blanco, D.~Zeillinger, D.~Wagner, A.~Fr\"uhbis-Kr\"uger.


\section{Setting}

\no The concepts and constructions that are successfully used to prove the embedded resolution of singularities over fields of characteristic zero require some amendments for their characteristic free definition. It remains an open problem whether these will suffice to give a proof of resolution in positive characteristic and arbitrary dimension. 

We shall work with complete regular local rings $R =(R,m_R)$ of dimension $n+1$ over an algebraically closed field $\K$, and of residue field $R/m_R\isom\K$. By Cohen's structure theorem, $R$ is a formal power series ring in $n+1$ variables over $\K$. It should be thought of as the completion of the local ring of some regular noetherian scheme $W$ over $\K$ at a closed point $a$, and ideals $J$ of $R$ as defining the formal neighborhood at $a$ of a closed subscheme $X$ of $W$. Typically, a regular system of parameters $(z,x)=(z,x_1,\ldots,x_n)$ will be chosen, with a distinguished parameter $z$. We then often fix a ring inclusion $\rho: R/(z)\, \isom\, Q\subset R$ providing a section of the projection $R\map R/(z)$, for some subring $Q$ of $R$. The induced isomorphism $R\isom Q[[z]]$ will be used frequently. 

The {\it order} of an ideal $J$ in $R$ is defined as $\ord\, J=\ord_{m_R}J=\sup\{k\in\N,\, J\subset m_R^k\}$. If $P$ is a prime ideal of $R$, we define the order $\ord_PJ$ of $J$ with respect to $P$ as the order of $J\cdot R_P$ in the localization $R_P$. For regular ideals $P$, it equals $\sup\{k\in\N,\, J\subset P^k\}$. A closed subscheme $Z=V(P)$ of $\Spec(R)$ is said to be {\it contained in the equimultiple locus of $J$} if $\ord_PJ=\ord\, J$ holds.  


The {\it initial form} $\inin(f)$ of an element $f\in R$ is the homogeneous polynomial of lowest degree of the power series expansion of $f$ with respect to the $m_R$-adic filtration of $R$. If $z,x_1,\ldots,x_n$ are given regular parameters in $R$ and \hh $w\in \mathbb Q$ is a rational number $\geq 1$, we define for $f\in R$ with expansion $f(z,x)=\sum_{i\geq 0} f_i z^i$ and coefficients $f_i\in\K[[x_1,\ldots,x_n]]$ the {\it weighted order} $\ord_wf$ of $f$ with respect to the weight vector $(w,1,\ldots,1)$ as the minimum of the values $wi+\ord\, f_i$, the order of $f_i$ being taken in $\K[[x_1,\ldots,x_n]]$. It clearly only depends on the choice of $z$. The {\it weighted initial form} $\inin_w(f)$ of $f$ with respect to the weight vector $(w,1,\ldots,1)$ (and the parameters $(z,x)$) is then defined as the sum $\inin_w(f)=\sum\, \inin(f_i)z^i$, where the sum ranges over those $i$ for which the minimal value of $wi+\ord\, f_i$ is attained.


Let $\pi: R\map R'$ be a {\it completed local blowup} of $R$, with regular center $Z=V(P)$ in $W=\Spec(R)$, for some ideal $P$ of $R$ and a complete regular local ring $R'$. By this we understand that $R'$ is the completion of a local ring $\OO_{W',a'}$ where $W'=\Proj(\oplus_{i\geq0}P^i)$ is the blowup of $W$ in $Z$, $a'\in W'$ is a closed point and $\pi:R\map R'$ is the induced map of complete local rings. As $R$ and $Z$ are regular, $R'$ is again regular, and actually isomorphic to $R$. Occasionally we shall identify $R'$ with $R$.

The {\it weak transform} of an ideal $J$ of $R$ under $\pi$ is defined as the (unique) ideal $J'$ of $R'$ so that \med

\cl{$\pi(J)\cdot R' = x_1^{\ord_PJ} \cdot J'$,}\med

\no where $x_1\in R'$ defines the exceptional component $E$ of $\pi$ in $\Spec(R')$. It is well known that under blowups in regular centers contained in the equimultiple locus of $J$ the order of $J$ does not increase when passing to $J'$ \cite{Ha_Obergurgl_2014}.


Define the {\it coefficient ideal} $\KK=\coeff_{V\rrho}(J)$ of $J$ with respect to a regular hypersurface $V=V(z)$ in $\Spec(R)$ and a section $\rho:R/(z)\isom Q\subset R$ of $R\map R/(z)$ as the ideal of $Q$ defined by\med

\cl{$\coeff_{V\rrho}(J)=\sum_{i<c} (f_i,\, f\in J)^{c!\over c-i}$,}\med

\no where $c=\ord\, J$ and elements $f\in J$ are expanded as series $f=\sum_{i\geq 0} f_iz^i$ in $Q[[z]]$, with coefficients $f_i$ in $Q$. The coefficient ideal depends on $V$ and $\rho$, but not on the choice of the parameter $z$ defining $V$. By abuse of notation, we suppress the dependence of the coefficient ideal on the choice of the section $\rho$. This does no harm in our context since the order of the coefficient ideal (which is our main concern) only depends on $V$ and not on $\rho$. 


Let $V=V(z)$ be a regular hypersurface in $\Spec(R)$ and let $o$ be the order of the coefficient ideal $\coeff_V(J)$. Further, set $w=\frac{o}{c!}$ where $c$ is the order of the ideal $J$. Then the minimum of the weighted orders $\ord_w(f)$ of elements $f\in J$ with respect to the regular parameter $z$ and the weight vector $(w,1,\ldots,1)$ equals $c\cdot w$.


A regular hypersurface $V=V(z)$ in $\Spec(R)$ has {\it weak maximal contact with $J$} if the order of the coefficient ideal $\coeff_{V\rrho}(J)$ of $J$ with respect to $V$ is maximized over all choices of regular hypersurfaces in $\Spec(R)$ and if for any blowup with regular center $Z=V(P)$ contained in $V$ and in the equimultiple locus of $J$, the strict transform of $V$ in $\Proj(\oplus_{i\geq0}P^i)$ contains all points at which the order of the weak transform of $J$ has remained constant.

Two cases can occur: The supremum of the orders of $\coeff_{V\rrho}(J)$ over all $V$ may be infinite, in which case $J$ is of the form $J=(z^c)$ for some regular parameter $z\in R$, and has trivial coefficient ideal equal to $0$ with respect to $V=V(z)$. This case is irrelevant for our investigations and will be discarded. Or, the supremum of the orders is bounded, in which case the maximum exists and is realized by some $V$. Such a $V$ can then be chosen so that its strict transform contains all points where the order of the weak transform of $J$ has remained constant. If the characteristic is zero, then $V$ can even be chosen in a way so that it has maximal contact with $J$.


Let $D$ be a (not necessarily reduced) normal crossings divisor in $\Spec(R)$, and let $J\subset R$ be an ideal. A regular hypersurface $V=V(z)$ is {\it compatible} with $D$ and $J$ if it has normal crossings with $D$ and if there is a section $\rho:R/(z)\isom Q\subset R$ of $R\map R/(z)$ so that the coefficient ideal $\KK=\coeff_{V\rrho}(J)$ of $J$ with respect to $V$ and $\rho$ factors into $\KK=M\cdot I$, for some ideal $I$ of $Q$, where $M$ is the principal ideal of $Q$ defining $D\cap V$ in $V$. 


Let $J$ be an ideal and $D$ a normal crossings divisor for which there exists a regular hypersurface $V$ that has weak maximal contact with $J$ and is compatible with $D$. The {\it residual order} of such an ideal $J$ with respect to $D$ is defined as \med

\cl{$ $residual-order$_D(J)=\ord(\coeff_{V\rrho}(J))- \ord\, M=\ord\, I$,}\med

\no where $V=V(z)\subset\Spec(R)$ is a hypersurface of weak maximal contact with $J$ and compatible with $D$, and where $M$ is the ideal which defines $D\cap V$ in $V$ and appears in the factorization $\coeff_{V}(J)=M\cdot I$. Notice that the residual order is independent of the choice of $V$. This numeral is frequently used in the proof of resolution of singularities in characteristic zero. It is supposed to measure the ``distance'' of $\KK$ from being a principal monomial ideal supported by $D$.


A completed local blowup $\pi:R\map R'$ with center $Z=V(P)$ is said to be {\it permissible} with respect to $J$ and $D$ if the center $Z$ of $\pi$ is regular, has normal crossings with $D$ and if there exists a hypersurface $V$ of weak maximal contact with $J$, compatible with $D$, and such that $Z$ is contained in $V$ and in the equimultiple loci of $J$ and $I$; here $I$ is defined through $\coeff_V(J)=M\cdot I$ as before. 


The {\it transform} $D'$ of $D$ with respect to $J$ under a permissible completed local blowup $\pi:R\to R'$ is defined as the normal crossings divisor $D'=D^s + (\ord_P\KK -c!)\cdot E$ in $\Spec(R')$, where $D^s$ denotes the strict transform of $D$ and $E=\pi^{-1}(Z)$ is the new exceptional component (cf. \cite{EH}). Here, $c$ is the order of $J$ in $R$, and $\KK$ is the coefficient ideal of $J$ with respect to a hypersurface of weak maximal contact $V$ with $J$ and compatible with $D$. The definition of $D'$ is independent of the choice of the hypersurface $V$.

If $J$ and $D$ admit a regular hypersurface $V$ having weak maximal contact with $J$ and compatible with $D$, it can be shown that there exists, for every permissible completed local blowup $\pi:R\to R'$ under which the order of $J$ remains constant, a regular hypersurface $U'$ in $\Spec(R')$ which has weak maximal contact with the weak transform $J'$ of $J$ and is compatible with $D'$ (cf. the proof of the proposition below). If the characteristic is zero, then the hypersurface $V$ in $\Spec(R)$  can be chosen in such a way that its strict transform $V'$ in $\Spec(R')$ has these properties. This is no longer true over fields of positive characteristic.


Regular parameters $(z,x)$ in $R$ are called {\it subordinate} to a permissible blowup $\pi$, an ideal $J$, a normal crossings divisor $D$ and a hypersurface $V$ of weak maximal contact with $J$ and compatible with $D$, if $V=V(z)$, the components of $D$ are supported by the hypersurfaces $V(x_i)$ of $\Spec(R)$, and if the defining ideal $P$ of the center $Z$ is generated by $z$ and $x_i$, for $i$ varying in a subset $S$ of $\{1,\ldots,n\}$. Permuting the $x_i$ if necessary, we may assume that the blowup occurs in the $x_1$-chart. There then exist a subset $T$ of $S$ containing $1$ and constants $t_i\in \K^*$, for $i\in\TT\sm\{1\}$, so that $\pi$ is defined by
 \[\begin{array}{ll}
 z\map x_1z, &\\ 
 x_1\map x_1, &\\
 x_i\map x_1(x_i+t_i) & \text{for $i\in \TT\sm \{1\}$,}\\
 x_i\map x_1x_i & \text{for $i\in S\setminus T$,}\\
 x_i\map x_i & \text{for $i\notin S$,}
\end{array}\]
where $R'$ is identified with $R$ and $(z,x_1,\ldots,x_n)$ denotes \hh a regular system of parameters in $R$ and $R'$. We may further assume that either $T=\{1\}$ or that for all indices $i\in T$ the inclusion $V(x_i)\subset D$ holds.\med


If the characteristic of $\kk$ is zero, it is well-known that for all permissible completed local blowups $\pi:R\to R'$ under which the order of $J$ remains constant $\ord J'=\ord J$ when passing to its weak transform $J'$, the residual order does not increase, i.e., 
\[\textnormal{residual-order}_{D'}J'\leq\textnormal{residual-order}_DJ\]
holds. Over fields of positive characteristic, this is no longer true: the residual order may increase.

\gb

\section{Results}

\no The characterization of ideals and permissible blowups for which the residual order increases goes as follows.

\begin{unnumberedtheorem}

Let $R$ be a complete regular local noetherian ring $R$ of dimension $n+1$ over an algebraically closed field $\K$ of positive characteristic $p>0$. Let $D$ be a normal crossings divisor in $\Spec(R)$. Let be given an ideal $J$ in $R$ admitting a hypersurface of weak maximal contact and compatible with $D$. Let $\pi: R\map R'$ be a completed local blowup of $R$, permissible with respect to $J$ and $D$ with center $Z$ defined by the ideal $P$ of $R$. Denote by $J'$ the weak transform of $J$ in $R'$, and by $D'$ the transform of $D$ in $\Spec(R')$ with respect to $J$.
  
Let $V$ in $\Spec(R)$ be a hypersurface of weak maximal contact with $J$ and com\-patible with $D$ such that $Z$ is contained in $V$ and in the equimultiple loci of $J$ and $I$, where $I$ appears in the factorization $\coeff_V(J)=M\cdot I$, with $M$ the ideal in $V$ defining $D\cap V$. Choose regular parameters $(z,x)=(z,x_1,\ldots,x_n)$ of $R$ subordinate to $\pi$, $J$, $D$ and $V$, and let $T\subset S\subset\{1,...,n\}$ and $t_i\in\K^*$ be as above. 

Assume that $J'$ has the same order as $J$ but that its residual order with respect to $D'$ is {\rm larger} than the residual order of $J$ with respect to $D$. Then the following conditions must be satisfied. 

\begin{enumerate} 


\item The order $c$ of $J$ is a multiple $c=m\cdot p^e$ of a $p$-th power, with $m\geq 1$ not divisible by $p$ and $e\geq 1$.\med


\item The order $o$ of the coefficient ideal $\KK$ of $J$ with respect to $V$ is a multiple $o=\w\cdot c!$ of $c!$, with $\w\geq 2$.\med


\item There exists a homogeneous, non $p^e$-th power polynomial $F$ 
in $x_1,\ldots,x_n$ of degree $w\cdot p^e$ so that the weighted initial form $\inin_w(f)$ with respect to $(z,x)$ and the weights $(\w, 1,\ldots,1)$ of every element $f\in J$ of minimal weighted order $c\cdot w$ is the $m$-th power of a purely inseparable polynomial, say\med

\cl{$\inin_w(f)=\alpha\cdot (z^{p^e}+F(x))^m$,}\med

\no for some non-zero constant $\alpha\in \K^*$.\med




\item Factorize $F$ into $F(x)=x^r\cdot G(x)$ with $r_i=\ord_{(x_i)}F$, for $i\in \TT$, and $G$ a homogeneous polynomial of degree $\u=\deg F-\sum_{i\in \TT} r_i$.  If $Q_\TT$ denotes the ideal of $\K[[x_1,\ldots,x_n]]$ generated by  $x_i-t_ix_1$, for $i\in\TT\sm\{1\}$, and $x_i$, for $i\not\in \TT$, then\med

\hh \cl{$\ord^{{\rm mod}\, p^e}_{Q_\TT}F > \u$,}\med

\no where $\ord^{{\rm mod}\, p^e}_{Q_\TT}F$ denotes the maximum of the orders $\ord_{Q_\TT}(F + H^{p^e})$ over all polynomials $H$ in $x_1,\ldots,x_n$.\med

\end{enumerate}

\no The inequality \hh $\ord^{{\rm mod}\, p^e}_{Q_\TT}F> \u$ from (4) implies the following conditions (5) to (9). Let $\ell<e$ be the largest integer so that $F$ is a $p^\ell$-th power, and denote by $b$ the number of exponents $r_i$, for $i\in\TT$, not congruent to \hh $0$ modulo $p^{\ell+1}$.\med

\begin{enumerate} \setcounter{enumi}{4}

\item Denote by $\tt$ the vector in $\K^n$ of components $t_i$ for $i\in \TT\sm\{1\}$, and $0$ otherwise. The polynomial $G(x)$ of the factorization $F(x)=x^r\cdot G(x)$ has, up to $p^e$-th powers, a unique form, \med 

\cl{$G((1,x_2,\ldots,x_n)+\tt)= \lfloor \prod_{i\in \TT\setminus\{1\}}(x_i+t_i)^{-r_i}\cdot N^{p^e}(x_2,\ldots,x_n)\rfloor_{\u}$}\med

\no for some polynomial $N(x_2,\ldots,x_n)$. 
Here, the product $\prod_{i\in{\TT\setminus\{1\}}} (x_i+t_i)^{-r_i}$ is considered as a power series, and $\lfloor -\rfloor_{\u}$ denotes the $\u$-jet of a power series. 


\item The residues $0\leq \ol r_i<p^{\ell+1}$ of $r_i$ modulo $p^{\ell+1}$ satisfy the arithmetic inequality\med

\cl{$\ds\sum_{i\in\TT} \ol r_i\leq (b-1)\cdot p^{\ell+1}$.}\med

\no Equivalently, one has \med

\cl{$\ds \sum_{i\in \TT} \ol r_i+\ol v\not=b\cdot p^{\ell+1}$.}\med


\item  For $j\not\in\TT$, the variables $x_j$ appear only as $p^e$-th powers in $F(x)$, say \med

\cl{$F(x)\in \K[x_i^{p^\ell}, x_j^{p^e},\, i\in \TT,\, j\not\in\TT]$.}\med


\item For $i\in\TT$, the $p^\ell$-th logarithmic Hasse derivatives of $F(x)$ with respect to $x_i$ are of the form\med

\cl{$\ds x_i^{p^\ell}\cdot \6_{x_i^{p^\ell}}\, F(x) = x^r\cdot H_i(x)$,}\med

\no where $H_i$ is a polynomial in $(x_j-t_jx_1)^{p^\ell}$ and  $x_k^{p^e}$, for $j\in \TT\setminus\{1\}$ and $k\not\in\TT$.\med


\item The increase of the residual order is bounded by\med

\cl{{\rm residual-order}$_{D'}J'\leq{}${\rm residual-order}$_DJ+\frac{c!}{p}$.}\med

\end{enumerate}

\end{unnumberedtheorem}


\no{\bf Comments.} The statements of the theorem crystallize a broader background which will be explained below.

(a) The theorem only tells us something about the exceptional multiplicities and the (weighted) tangent cone of the ideal $J$. It does not make any statement about the higher order terms of the elements of $J$.

(b) The multiplicity of the new exceptional component in $D'$ equals $\ord_P\KK -c!$ and is hence a multiple of $c!$. Let $x_1$ be the parameter defining this component. Then the center $Z'= V(z,x_1)$ in $\Spec(R')$ is contained in the equimultiple locus of $J'$, has normal crossings with $D'$ and can be blown up until the exceptional multiplicity of this component has dropped to $0$.  

(c) The residual order is a questionable resolution invariant as is exhibited by an example of an infinite sequence of permissible blowups where the residual order tends to infinity \st \cite{Indefinite_Increase}. In this sequence, however, the centers are not chosen of maximal dimension, so this is not yet a counterexample to the resolution of singularities in positive characteristic.

(d) The increase of the residual order can only happen if under the blowup at least two components of $D$ are lost when passing to the reference point $a'$ in the new exceptional component.  


(e) The increase of the residual order represents a serious obstacle for trying to transfer the proof of resolution of singularities in characteristic zero to positive characteristic. For surfaces, it can still be used, but has to be modified slightly so as to perform appropriately under blowup, see \cite{HW, HP_Surfaces}. Already for three-folds the situation is unclear and no efficient resolution invariant (for embedded resolution) seems to be known (for the non-embedded case, see \cite{Abhyankar_3_Folds, Cossart_Piltant_1, Cossart_Piltant_2, Cutkosky_3_Folds}).

(f) For a fixed prime number $p$, the arithmetic inequality for the residues of the multiplicities $r_i$ in assertion (6) of the theorem always holds when $T$ contains sufficiently many indices $i$ with $r_i\not\congruent p^{\ell+1}$.

(g) For fixed numbers $n$, $p$, $e$ and $\ell$ as in the theorem, a homogeneous polynomial $F(x)=x^r\cdot G(x)$ of degree divisible by $p^e$, but not a $p^e$-th power, defines via $f=z^{p^e}+F(x)$ a weighted homogeneous hypersurface singularity whose residual order increases under blowup if and only if $G(x)$ is of the form specified in assertion (5) and the multiplicities $r_i$ fulfill the arithmetic inequality in assertion (6) of the theorem. 

(h) Assertion (2) and the bound in (9) have been known to Moh in the case of a purely inseparable hypersurface singularity \cite{Moh}.







\big


\section{Auxiliary Results}  

\no The proof of the theorem will rely on the following more technical result.


\begin{proposition} 

In the situation of the theorem, there exists an automorphism of $R$ sending $z$ onto $\z=z-q$, for some $q\in Q$ of order $\ord q\geq\frac{o}{c!}>1$, and inducing the identity on $Q$, so that $U=V(u)$ has again weak maximal contact with $J$ (but may no longer be compatible with $D$), and so that the following two conditions are satisfied.

\begin{enumerate}

 \item  The strict transform $\Vp$ of $\V$ has weak maximal contact with $J'$ and is compatible with $D'$.


 \item Factorize the coefficient ideal $\KK_1'=\coeff_{\Vp}(J')$ of $J'$ with respect to $\Vp$ and $\sigma_1': R'/(\zp)\isom Q\subset R'$ into $\KK_1'=M_1'\cdot I_1'$ with $M_1'$ the principal monomial ideal defining $D'\cap \Vp$ in $\Vp$. Then the residual order $\ord\, I_1'$ of $J'$ with respect to $D'$ is bounded by \hh\med

\cl{$\ord\, I< \ord\, I_1'\leq \ord_{Q_\TT} \inin(h) - \sum_{i\not\in \TT} s_i$,}\med

\no for any element $h$ of minimal order $o$ of the coefficient ideal $\KK_1=\coeff_{\V}(J)$ of $J$ with respect to $\V$ and the ring inclusion $\sigma_1: R/(\z)\isom Q\subset R$, and where $s_i=\ord_{(x_i)}D$.

\end{enumerate}

\end{proposition}


To show this, we need two lemmata. Lemma \ref{coeff_ideal_lemma} will clarify how the orders of the coefficients $f_i$ in the expansion $f=\sum_{i\geq0}f_iz^i$ are related to the order of the coefficient ideal. In Lemma \ref{cleaning_lemma} we will investigate the effect of coordinate changes $\z=z-\g$ with $\g\in\kk[[\x]]$ on the coefficient ideal. In particular, we will see that if the coordinate change increases the order of the coefficient ideal with respect to $V(z)$, then the element $\g$ has to be \hh of a very specific form. 


\begin{lemma} \label{coeff_ideal_lemma}
 Let $R$ be the power series ring $\kk[[z,\x]]$ with $\x=(x_1,\ldots,x_n)$ for some field $\kk$. Denote the maximal ideal of $R$ by $\mm$. Let $J\subset R$ be an ideal of order $\ord J=c$. Let each element $f\in J$ have the expansion $f=\sum_{i\geq0}f_iz^i$ with $f_i\in\kk[[\x]]$.
 
 Set $\Jo=\coeffJ$ for $V=V(z)$ and a section $\rho:R/(z)\cong \QQ\subset R$ of $R\map R/(z)$. Define $o=\ord\Jo$ and $w=\oc$.
 
 Then the following statements hold:
 \begin{enumerate}
  \item The order of $\Jo$ can be expressed as 
  \[o=\min_{f\in J}\min_{i<c}\frac{c!}{c-i}\ord f_i.\]
  Consequently, for all elements $f\in J$ and indices $i<c$, the inequality
  \[\ord f_i\geq (c-i)w\]
  holds.
  \item $o\geq c!$.
  \item $o>c!$ holds if and only if $J\congruent (z^c)$ modulo $\mm^{c+1}$.
 \end{enumerate}
\end{lemma}


\begin{proof}
 Immediate from the definition of the coefficient ideal. 
\end{proof}


\begin{lemma} \label{cleaning_lemma}
 Let $R$ be the power series ring $\kk[[z,\x]]$ with $\x=(x_1,\ldots,x_n)$ where $\kk$ is a field of characteristic $p>0$. Consider a change of coordinates $\z=z-\g$ where  $\g\in \kk[\x]$ is a homogeneous polynomial, and define $V=V(z)$, $\V=V(\z)$. Let $J\subset R$ be any ideal of order $\ord J=c$ and let $p^e$ the largest $p$-th power dividing $c$.
 
Let $\Jo=\coeffJ$ and $\tJo=\coeff_{\V}(J)$, 
and set $o=\ord\Jo$, $o_1=\ord\tJo$, and $\w=\oc$. The following statements hold:
 
 \begin{enumerate}
  \item If $\deg \g \geq \w$, then $o_1\geq o$.
  \item If $\deg\g<\w$ and there exists an element $f\in J$ that is $z$-regular of order $c$, then $o_1=c!\cdot\deg\g<o$.
  \item Let $1\leq i\leq n$ be an index. If $\ord_{(x_i)}\g\geq \frac{1}{c!}\ord_{(x_i)}\Jo$, then $\ord_{(x_i)}\tJo\geq\ord_{(x_i)}\Jo$.
  \item Let $f\in J$ be an element that is $z$-regular of order $c$. Let $f$ have the expansion $f=\sum_{i\geq0}f_iz^i$ with $f_i\in\kk[[\x]]$. If $o_1>o$ holds, then $\deg \g=\w$ and $\g$ fulfills
  \[\g^\q=\lambda\cdot\init(f_{c-\q})\]
  for a non-zero constant $\lambda\in K^*$.
 \end{enumerate}
\end{lemma}


\begin{proof}
 Let each element $f\in J$ have expansions $f=\sum_{i\geq0}f_iz^i$ and $f=\sum_{i\geq0}\wt f_i\z^i$ with $f_i,\wt f_i\in\kk[[\x]]$. Then 
 \[\wt f_i=\sum_{k\geq i}\binom{k}{i}f_k\g^{k-i}.\]
 Notice that an element $f$ is $z$-regular of order $c$ if and only if the coefficient $f_c$ is a unit.
 
 Statements (1), (2) and (3) can be verified directly by using the formula for $\wt f_i$ and Lemma \ref{coeff_ideal_lemma} (1).
 
 To prove statement (4), let $f\in J$ be $z$-regular of order $c$. If $\deg \g>\w$, it is straightforward to show that $o_1=o$ holds. By statement (2) this implies that $\deg\g=\w$ has to hold.
 
 Assume now that there exists an index $c-\q<i<c$ such that $\ord f_i=(c-i)\w$. Let $i$ be maximal with this property. Notice that $\binom{c}{i}\equiv0\pmod p$ by Lucas' theorem on binomial coefficients in characteristic $p>0$. Using the form of $\wt f_i$, the maximality of $i$ and the fact that $\binom{c}{i}\equiv0\pmod p$, it follows that $\ord \wt f_i=(c-i)\w$. Consequently, $o_1\leq o$ holds by Lemma \ref{coeff_ideal_lemma} (1), contradicting the assumption.
 
 Hence, $\ord f_i>(c-i)\w$ holds for all indices $c-\q<i<c$. If we had $\ord f_{c-\q}>\q\w$, then
 \[\init(\wt f_{c-\q})=\binom{c}{\q}f_c(0)\g^\q.\]
 Since $\binom{c}{\q}\not\equiv0\pmod p$ by Lucas' theorem, this implies that $\ord \wt f_{c-\q}=\q\w$ and hence, $o_1\leq o$ by Lemma \ref{coeff_ideal_lemma} (1), again a contradiction.
 
 Thus, $\ord f_{c-\q}=\q\w$. Since $o_1>o$,  it follows that $\ord \wt f_{c-\q}>\q\w$. This gives 
 \[\init(f_{c-\q})+\binom{c}{\q}f_c(0)\g^\q=0,\]
 which proves the assertion. 
\end{proof}


\begin{proof}[Proof of the proposition:]
Define parameters $y=(y_1,\ldots,y_n)$ by setting
 \[y_i=\begin{cases} x_i-t_ix_1 & \text{for $i\in \To$,}\\
		    x_i & \text{otherwise.}
      \end{cases}\]
 Notice that the ideal $P$ is generated by the parameters $z$ and $y_i$ for $i\in S$. Also, $Q_T=(y_2,\ldots,y_n)$. Further, the map $\pi:R\to R'$ has the following simple form with respect to $(z,y)$: $z\to x_1z$, $y_1\to x_1$, $y_i\to x_1x_i$ for $i\in S\setminus\{1\}$ and $y_i\to x_i$ for $i\notin S$. One also says that the blowup map $\pi:R\to R'$ is monomial with respect to the parameters $(z,y)$. Monomial blowup maps have the advantage that they make calculations in coordinates particularly easy.

 We will begin with assertion (1).  To this end, let us first verify that $J\equiv (z^c)$ modulo $\mm^{c+1}$.
 
 Recall that the strict transform $V'=V(z)$ of $V$ in $\Spec(R')$ is non-empty since $\ord J'=\ord J$ holds and $V$ has weak maximal contact with $J$. Further, it is easy to see that the regular hypersurface $V'$ is compatible with $D'$.
 
 Recall that the center of blowup $Z$ is contained in the equimultiple locus of $I$. Thus, it \hh is easy to see that the ideal $I'$ in the factorization $\coeff_{V'}(J')=M'\cdot I'$ fulfills $\ord I'\leq\ord I$.  Since we assumed that the residual order increases under the local blowup $\pi$, we conclude from this that $V'$ does not have weak maximal contact with $J'$.
 
 Assume now that $J\not\equiv(z^c)$ modulo $\mm^{c+1}$. Since $V$ has weak maximal contact with $J$, this implies by Lemma \ref{coeff_ideal_lemma} (3) that $\ord\coeff_{\wt V}(J)=c!$ holds for all regular hypersurfaces $\wt V\subset\Spec(R)$. Consequently, by Lemma \ref{coeff_ideal_lemma} (3) there is no regular parameter $u\in R$ for which $J\equiv(u^c)$ modulo $\mm^{c+1}$ holds. It is straightforward to verify that this implies that there is also no regular paramter $\wt u\in R'$ for which $J'\equiv(\wt u^c)$ modulo $\mmp^{c+1}$ holds. This would imply by Lemma \ref{coeff_ideal_lemma} that any regular hypersurface $\Vpp\subset\Spec(R')$ has weak maximal contact with $J'$. Since we already know that $V'$ does not have weak maximal contact with $J'$, this is a contradiction. Hence, $J\equiv(z^c)$ modulo $\mm^{c+1}$ holds.
 
 It is immediate to see that this implies the existence of an element $f\in J$ which is $z$-regular of order $c$. Set $f'=x_1^{-c}\pi(f)\in J'$. Then the element $f'$ is also $z$-regular of order $c$. Let these elements have the expansions $f=\sum_{i\geq0}f_iz^i$ and $f'=\sum_{i\geq0}f_i'z^i$ with $f_i,f_i'\in\kk[[\x]]$.
 
 Let $\Vpp=V(\zpp)\subset\Spec(R')$ be a regular hypersurface which has weak maximal contact with $J'$. Set $\Kpp=\coeff_{\Vpp}(J')$. If the element $\zpp\in R'$ is not $z$-regular, then it is straightforward to verify with Lemma \ref{coeff_ideal_lemma} (1) and using the fact that $f$ is $z$-regular of order $c$ that $\ord \Kpp=c!$. This contradicts the fact that $V'$ does not have weak maximal contact with $J'$. Thus, we may assume that \hh $\zpp=z-\wt g$ for an element $\wt g\in\kk[[\x]]$.
 
 Set $K'=\coeff_{V'}(J')$. By Lemma \ref{cleaning_lemma} we know that $\ord \wt g=\frac{1}{c!}\ord K'$. Set $\gp=\init(\wt g)$. Further, define $\zp=z-\gp$, $\Vp=V(\zp)\subset\Spec(R')$ and $K_1'=\coeff_{\Vp}(J')$. Then it is clear by Lemma \ref{cleaning_lemma} that the chain of inequalities
 \[\ord \Kpp\geq \ord K_1'>\ord K'\]
 holds.
 We know by Lemma \ref{cleaning_lemma} (4) that
 \[\gp^\q=\lambda\cdot\init(f_{c-\q}')\]
 for some non-zero constant $\lambda$, where $p^e$ is the largest $p$-th power dividing $c$. Notice that $f_{c-\q}'=x_1^{-\q}\pi(f_{c-\q})$.
 
 It is clear that for all indices $i=1,\ldots,n$ the inequality
 \[\ord_{(x_i)}\gp\geq \frac{1}{\q}\ord_{(x_i)}f'_{c-\q}\geq\frac{1}{c!}\ord_{(x_i)}K'\]
 holds. Since $V'=V(z)$ is compatible with $D'$, this implies by Lemma \ref{cleaning_lemma} (3) that $\Vp$ is also compatible with $D'$.
 
 Since the map $\pi:R\to R'$ is monomial with respect to the parameters $(z,\y)$, it is straightforward to verify that there exists an element $\g\in R$ that fulfills $\gp=x_1^{-1}\pi(\g)$ and is of the form
 \[\g^\q=\lambda\cdot\init_\sigma(f_{c-\q})\] 
 where $\init_\sigma(f_{c-\q})$ denotes the weighted initial form of $f_{c-\q}$ with respect to the weights $\sigma(y_i)=2$ for $i\in S\setminus\{1\}$ and $\sigma(y_i)=1$ for all other indices $i$. Set $\z=z-\g$. Then
 \[\pi(\z)=x_1(z-\gp)=x_1\zp.\]
Hence, the hypersurface $\Vp$ is the strict transform of the hypersurface $\V=V(\z)$. Set $K_1=\coeff_{\V}(J)$.
 
 We will now show that $\V$ has weak maximal contact with $J$ and $K_1$ has a factorization
 \[\tJo=\Big(\prod_{i\notin \A}x_i^{s_i}\Big)\cdot I_1\]
 for some ideal $I_1$ of $Q$.
 Since \hh $\g^\q=\lambda\cdot \init_\sigma(f_{c-\q})$, it is clear that the inequality
 \[\ord \g\geq \frac{1}{\q}\ord f_{c-\q}\geq \frac{o}{c!}\]
 holds, as well as, for all indices $i\notin \A$,  the inequalities
 \[\ord_{(x_i)}\g\geq \frac{1}{\q}\ord_{(x_i)}f_{c-\q}\geq \frac{1}{c!}\ord_{(x_i)}K\]
(since $x_i=y_i$ for $i\notin\A$). By Lemma \ref{cleaning_lemma} (1) and (3) this implies that
 \[\ord \tJo\geq \ord \Jo=o\]
 holds and that $\tJo$ has the claimed factorization. Since the hypersurface $V$ already had weak maximal contact with $J$, we must have $\ord\tJo=o$. Hence, also $\V$ has weak maximal contact with $J$.
 
 
 We now consider two cases. If the hypersurface $\Vp$ has weak maximal contact with $J'$, we are done. On the other hand, if $\ord\Kpp>\ord K_1'$, we may replace $\wt g$ by $\wt g-\gp$ and repeat the whole argument. Since the order of $\Kpp$ is finite, we can thus construct hypersurfaces $\V$ and $\Vp$ with the claimed properties after finitely many iterations.
 
 To prove assertion (2), we rewrite the claimed inequality so as to allow a calculation in coordinates. Let $h\in\tJo$ be an element of minimal order $o$. Define for a tuple $\gamma=(\gamma_1,\ldots,\gamma_n)\in\N^n$ the number
 \[|\gamma|_S=\sum_{i\in S}\gamma_i.\]
 Set $s=(s_1,\ldots,s_n)$. Since $\Vp$ is compatible with $D'$, we know that $\tJop$ is of the form 
 \[\tJop=\Big(x_1^{\dd+|s|_S-c!}\prod_{i\notin\A}x_i^{s_i}\Big)\cdot I_1'.\]
 It is straightforward that $x_1^{-c!}\pi(h)\in\tJop$. This implies that
 \[\ddp=\ord \tJop-(\dd+|s|_S-c!)-\sum_{i\notin\A}s_i\]
 \[\leq\ord \pi(h)-\dd-|s|_S-\sum_{i\notin\A}s_i.\]
 Hence, it remains to verify that the inequality
 \[\ord\pi(h)\leq \ord_{Q_T}\init(h)+\dd+|s|_S\]
 holds. 
 Let $h$ have the following power series expansion with respect to $y$:
 \[h=\sum_{\alpha\in\N^n}c_\alpha \y^\alpha.\]
 There exists a multi-index $\beta\in\N^n$ such that $c_\beta\neq0$, $|\beta|=o$, and 
 \[\sum_{i\geq 2}\beta_i=\ord_{Q_T}\init(h).\]
 Since \hh \st $K_1=\coeff_{\V}(J)=\Big(\prod_{i\notin \A}x_i^{s_i}\Big)\cdot I_1$, we know that $\beta_i\geq s_i$ for all indices $i\notin\A$. Consequently,
 \[|\beta|_S=o-\sum_{i\notin S}\beta_i\leq o-\sum_{i\notin S}s_i=\dd+|s|_S.\] 
 Further,
 \[\pi(h)=\sum_{\alpha\in\N^n}c_\alpha x_1^{|\alpha|_S}\prod_{i\geq2}x_i^{\alpha_i}.\]
 From this we conclude that
 \[\ord\pi(h)\leq |\beta|_S+\sum_{i\geq 2}\beta_j\leq \ord_{Q_T}\init(h)+\dd+|s|_S.\]
 This proves the claimed inequality.
\end{proof}


\section{Proof of the Theorem}  

\no Using the bound for the order of $I_1'$ established in the proposition, we can now prove the theorem without considering the coefficient ideals of the weak transform $J'$ of $J$. Instead, we will directly work in $R$ with the coefficient ideals $\Jo$ and $\tJo$ of $J$.

\begin{proof}[Proof of the theorem:]
 Let $\g$, $\z$, $\tJo$ and $s_i$ be defined as in the proposition. Set $w=\frac{o}{c!}$. Let $p^e$ be the biggest $p$-th power dividing $c$ and set $m=\frac{c}{p^e}$. \st Notice that $m\equiv \binom{c}{p^e}\pmod p$. Recall that $\ord \g\geq w>1$.
 
 
 We begin with assertions (1), (2) and (3). Let $f\in J$ be an element of minimal weighted order $c\w$. Let the weighted initial form of $f$ have the expansion \hh $\init_\w(f)=\sum_{i\geq0}F_iz^i$ with respect to the coordinates $(x,z)$, with $F_i\in \kk[\x]$. Hence, either $F_i=0$ or $F_i$ is a homogeneous polynomial of degree $\deg F_i=(c-i)\w$ for all indices $i$. In particular, $F_i$ is $0$ for all indices $i>c$.
 
 Let $i<c$ be an index with $F_i\neq 0$. Due to the factorization $\Jo=M\cdot I$, the element $F_i$ has a factorization 
 \[F_i=\prod_{j=1}^nx_j^{m_j}\cdot G_i\]
 with $m_j\geq\cic s_j$ for some polynomial $G_i$. Consequently,
 \[\ord_{Q_T}F_i=\underbrace{\ord_{Q_T}\prod_{j\in\A}x_j^{m_j}}_{=0}+\ord_{Q_T}\prod_{j\notin\A}x_j^{m_j}\cdot G_i\]
 \[\leq \deg \prod_{j\notin\A}x_j^{m_j}\cdot G_i\]
 \[\leq\frac{c-i}{c!}\dds.\]
 Set $\G=\init(\g)$. Denote by $\init_{w,(u,x)}(f)$ the weighted initial form of $f$ with respect to the parameters $(u,x)$ and the weight vector $(w,1,\ldots,1)$. Let this weighted initial form have the expansion \hh $\init_{w,(u,x)}(f)=\sum_{i\geq0}\wF_i\z^i$ with $\wF_i\in\kk[\x]$. If $\ord\g=w$, then
 \[\wF_i=\sum_{i\leq k\leq c}\binom{k}{i} F_k\G^{k-i}.\]
 On the other hand, if $\ord\g>w$, then $\wF_i=F_i$.
 Notice that if $\wF_i\neq0$ holds for an index $i<c$, then $\wF_i^\frac{c!}{c-i}$ is the initial form of an element of minimal order of $\tJo$. Hence, we know by the proposition \hh and the basic assumption $\ord I_1'>\ord I$ that 
\hh \begin{equation}\ord_{Q_T} \wF_i\geq \cic{\Big(\dd_1'+\rs\Big)}>\cic{\Big(\dd+\rs\Big)}. \tag{$\ast$} \label{eqwtFi}\end{equation}
 Consequently, for all indices $i<c$, either $F_i=0$ or $\wF_i\neq F_i$ holds. 
 
 This implies that $\ord\g=w$. Consequently, $w$ is an integer $w\geq 2$ and we have proved (2).
 
 We show that $F_c\neq 0$. Assume the contrary, and let $i<c$ be maximal with $F_i\neq0$. Then $\wF_i=F_i$ would hold by the formula for $\wF_i$, contradiction. This shows $F_c\neq0$. Thus, $f$ is $z$-regular of order $c$. After multiplication with a constant, we may assume that $F_c=1$.
 
 Assume that there exists an index $i<c$ that is not divisible by $\q$ such that $F_i\neq0$. Let $i$ be maximal with this property. Then it follows that
 \[\wF_i=F_i+\sum_{\substack{i<k\leq c\\ \q\mid k}}\underbrace{\binom{k}{i}}_{\equiv0}F_k\G^{k-i}+\sum_{\substack{i<k<c\\ \q\nmid k}}\binom{k}{i}\underbrace{F_k}_{=0}\G^{k-i}=F_i,\]
contradiction. Therefore $F_i$ is $0$ for all indices $i$ that are not divisible by $\q$.
 Set 
 \[F=\binom{c}{\q}^{-1}F_{c-\q}\]
 and recall that $\binom{c}{\q}\not\equiv0\pmod p$ by Lucas' theorem.
 Since $\wF_{c-\q}=\binom{c}{\q}(F+\G^\q)$, we get from (\ref{eqwtFi}) the inequality
 \begin{equation}\ord_{Q_T}(F+\G^\q)>\qc \dds. \tag{$\ast\ast$} \label{eqwtQT}\end{equation}
 Next, we establish for all indices \hh $i<\m=\frac{c}{p^e}$ the equality
 \[F_{i\q}=\binom{\m}{i}F^{\m-i}.\]
  Notice that the equality holds by definition for $i=\m-1$. Let $i<\m-1$ be maximal with $F_{i\q}\neq \binom{\m}{i}F^{\m-i}$. Then 
 \[\wF_{i\q}=F_{i\q}+\sum_{i<k\leq \m}\underbrace{\binom{k\q}{i\q}}_{\equiv\binom{k}{i}}F_{k\q}\G^{\q(k-i)}\]
 \[=F_{i\q}+\sum_{i<k\leq \m}\underbrace{\binom{k}{i}\binom{\m}{k}}_{=\binom{\m}{i}\binom{\m-i}{k-i}}F^{\m-k}\G^{\q(k-i)}\]
\hh  \[=F_{i\q}-\binom{\m}{i}F^{\m-i}+\binom{\m}{i}(F+H^\q)^{\m-i}.\]
 Together with the inequalities (\ref{eqwtFi}) and (\ref{eqwtQT}), this implies that \hh
 \[\ord_{Q_T}\Big(F_{i\q}-\binom{\m}{i}F^{\m-i}\Big)>(\m-i)\frac{\q}{ c!}\dds.\]
 But the factorization $\Jo=M\cdot I$ implies that
\hh  \hh \[\ord_{x_j}\Big(F_{i\q}-\binom{\m}{i}F^{\m-i}\Big)\geq(\m-i)\frac{\q}{ c!}s_j\]
 holds for all indices $j\in\A$. It follows that \hh \hh \st
 \[\ord_{Q_T}\Big(F_{i\q}-\binom{\m}{i}F^{\m-i}\Big)\leq \ord\Big(F_{i\q}-\binom{\m}{i}F^{\m-i}\Big)-(\m-i)\frac{\q}{ c!}\sum_{j\in\A}s_j\]
 \[= (\m-i)\frac{\q}{ c!}\dds,\]
 which contradicts the above inequality. Therefore
 \[\init_\w(f)=(z^\q+F)^{\m}\]
 holds. 
 
 Let now $h\in J$ be another element of weighted order $c\w$. Let $h$ have the expansion $h=\sum_{i\geq0}h_iz^i$. Using the same argument as before, we know that $h$ is $z$-regular of order $c$. Hence, $h_c$ is a unit. Consider the element $h-h_c(0)\cdot f\in J$. Since this element is not $z$-regular of order $c$, we know that its weighted order is strictly larger than $c\w$. This implies that
 \[\init_\w(h)=h_c(0)\cdot\init_\w(f)=h_c(0)\cdot (z^\q+F)^{\m}.\]
 Assume that $F$ is a $\q$-th power. Set $z_1=z+F^\frac{1}{\q}$ and let $V_1$ be the regular hypersurface $V_1=V(z_1)$ in $\Spec(R)$. Then it is straightforward to verify that $\ord\coeff_{V_1}(J)>o$. Thus, the hypersurface $V$ would not have had weak maximal contact with $J$, contradiction. So assertion (3) is shown. But as $F$ is not a $\q$-th power we must have $e\geq 1$, thus proving also (1).
 
 We can now easily prove assertion (4). Notice that
\hh  \[\u=\deg G\leq \frac{\q}{c!}\dds.\]
 Using inequality (\ref{eqwtQT}), this implies that \hh
 \[\ord_{Q_T}^{\mod p^e}F\geq \ord_{Q_T}(F+\G^\q)>\u.\]
 This proves (4).
 
To prove assertion (5), rewrite (4) as
 \[\ord_{(x_2,\ldots,x_n)}F((x_1,\ldots,x_n)+\tt x_1)+\G((x_1,\ldots,x_n)+\tt x_1)^\q>\u.\]
 Setting $x_1=1$, this is equivalent to
 \[\ord F((1,x_2,\ldots,x_n)+\tt)+\G((1,x_2,\ldots,x_n)+\tt)^\q>\u.\]
 Hence,
 \[F((1,x_2,\ldots,x_n)+\tt)+\G((1,x_2,\ldots,x_n)+\tt)^\q\in (x_2,\ldots,x_n)^{\u+1}.\]
 Set $N=\G((1,x_2,\ldots,x_n)+\tt)$. Since $F=x^r\cdot G$, we get that
 \[G((1,x_2,\ldots,x_n)+\tt)-\prod_{i\in\To}(x_i+t_i)^{-r_i}\cdot N^\q(x_2,\ldots,x_n)\in (x_2,\ldots,x_n)^{\u+1}.\]
 But since $\deg G((1,x_2,\ldots,x_n)+\tt)\leq \u$, this implies that
 \[G((1,x_2,\ldots,x_n)+\tt)=\lfloor\prod_{i\in\To}(x_i+t_i)^{-r_i}\cdot N^\q(x_2,\ldots,x_n)\rfloor_{\u}\]
as claimed.

 We proceed with assertion (6). \st First we verify that the two inequalities in the statement are equivalent. By definition of $v$, we have that $\deg F = \sum_{i\in T} r_i + v$. Further, we know that $\deg F$ is a multiple of $p^e$. Hence, it also a multiple of $p^{\ell+1}$. By definition of the number $b$, the residues $\ol{r_i}$ and $\ol{v}$ satisfy the inequalities 

 \[\sum_{i\in T} \ol{r_i} < b\cdot p^{\ell+1}\]
 and
 \[\sum_{i\in T} \ol{r_i} + \ol{v} \leq b\cdot p^{\ell+1}.\]

Since $0 \leq \ol{v} < p^{\ell+1}$, the following two are equivalent:

 \[\sum_{i\in T} \ol{r_i} > (b-1)\cdot p^{\ell+1}\]
 and
 \[\sum_{i\in T} \ol{r_i} + \ol{v} = b\cdot p^{\ell+1}.\]
 
 This proves that the two inequalities in assertion (6) are indeed equivalent. Now assume that they are violated and hence, the equality
 \[\sum_{i\in\A}\ol{r_i}+\ol{\u}=b\cdot p^{\ell+1}\]
 holds. Computation gives
 \[\prod_{i\in\To}(x_i+t_i)^{-r_i}=\prod_{i\in\To}(x_i+t_i)^{\ol{-r_i}}\cdot L^{p^{\ell+1}}(x_2,\ldots,x_n)\]
 for some element $L\in\kk[[x_2,\ldots,x_n]]$. Notice that
 \[\deg \prod_{i\in\To}(x_i+t_i)^{\ol{-r_i}}\leq \sum_{i\in T}\ol{-r_i}\]
 \[=b\cdot p^{\ell+1}-\sum_{i\in T}\ol{r_i}=\ol{\u}.\]
By (5) this implies that
 \[G((1,x_2,\ldots,x_n)+\tt)=\prod_{i\in\To}(x_i+t_i)^{\ol{-r_i}}\cdot \wt N^{p^{\ell+1}}(x_2,\ldots,x_n)\]
 for some polynomial $\wt N$. Consequently,
 \[\prod_{i\in\To}(x_i+t_i)^{r_i}\cdot G((1,x_2,\ldots,x_n)+\tt)\]
 is a $p^{\ell+1}$-th power. Since the degree of $F$ is divisible by $p^{\ell+1}$, also $F=x^r\cdot G$ is a $p^{\ell+1}$-th power, which contradicts the minimality of $p^{\ell}$ and proves (6).
 
It remains to prove assertions (7), (8) and (9). Let $k$ be an integer in the range $0\leq k<e$. Let $i\in\{1,\ldots,n\}$ be an index and assume that
 \[\partial_{x_i^{p^k}}(F)\neq0.\]
 Notice that $\partial_{x_i^{p^k}}(\G^\q)=0$. Consequently,
 \[\ord_{Q_T}\wF_{c-\q}=\ord_{Q_T}(F+\G^\q)\leq\ord_{Q_T}\partial_{x_i^{p^k}}(F)+p^k.\]
 If $i\in \A$, then
 \[\partial_{x_i^{p^k}}(F)=x_i^{r_i-p^k}\prod_{\substack{j\in\A\\j\neq i}}x_j^{r_j}H_{i,k}\]
holds for a homogeneous polynomial \hh $H_{i,k}\in \kk[\x]$ with $\deg H_{i,k}=\deg G=\u$. 
  On the other hand, if $i\notin\A$, then 
 \[\partial_{x_i^{p^k}}(F)=\prod_{j\in\A}x_j^{r_j}H_{i,k}\]
 for a homogeneous polynomial $H_{i,k}\in \kk[\x]$ with \hh $\deg H_{i,k}=\u-p^k$.
Together,
 \[\ord_{Q_T}\wF_{c-\q}\leq \ord_{Q_T}H_{i,k}+p^k\]
\hh  \[\leq\deg H_{i,k}+p^k=\u+\eps_{i,k}\]
 where 
 \[\eps_{i,k}=\begin{cases}
         p^k & \text{if $i\in\A$,}\\
         0 & \text{if $i\notin\A$.}
        \end{cases}
\]
 This proves, \hh together with the first inequality in (\ref{eqwtFi}), the following \hh inequalities: 
 \[\ddp\leq \frac{c!}{\q}\ord_{Q_T}\wF_{c-\q}-\sum_{i\not\in \TT} s_i\]
\hh  \[\leq\frac{c!}{\q}(\u+\eps_{i,k})-\sum_{i\not\in \TT} s_i\]
 \[\leq\ord I+\frac{c!}{\q}\eps_{i,k}.\]
 If $i\notin\A$, this implies that $\ddp\leq\dd$. Therefore,
 \[\partial_{x_i^{p^k}}(F)=0\]
holds for all indices $i\notin\A$ and all $k\geq0$ with $p^k<\q$. Thus, the variables $x_i$ with $i\notin\A$ only appear as $\q$-th powers in $F$. This proves (7).
 
\hh Recall that $\ell$ was chosen maximal such that $F$ is a $p^\ell$-th power. It is clear that this implies the existence of an index $i\in\A$ such that $\partial_{x_i^{p^\ell}}(F)\neq0$. The inequality above shows
 \[\ord I_1'\leq\ord I+\frac{c!}{\q}p^{\ell}\]
 \[\leq \ord I+\frac{c!}{p},\]
which proves (9).
 
 To prove assertion (8), fix an index $i\in \A$. Set $H_i=H_{i,\ell}$. \st Notice that the equality
 \[x_i^{p^\ell}\cdot\partial_{x_i^{p^\ell}}(F)=x^r\cdot H_i \]
 holds by definition of $H_i$. Further, we know that $H_i$ is a $p^\ell$-th power since $F$ is a $p^\ell$-th power by assumption. Hence, the assertion that $H_i$ is a polynomial in $(x_j-t_jx_1)^{p^\ell}$ and  $x_k^{p^e}$, for $j\in \TT\setminus\{1\}$ and $k\not\in\TT$, is equivalent to the equality
 \[\ord_{Q_T}H_i=\deg H_i.\]

 So assume that $\ord_{Q_T}H_i<\deg H_i$ holds. Since $H_i$ is a $p^\ell$-th power, this implies that
 \[\ord_{Q_T}H_i\leq \deg H_i-p^\ell.\]
 Plugging this into the chain of inequalities which we used to prove (9), we get that
 \[\ord_{Q_T}\wF_{c-\q}\leq \ord_{Q_T}H_{i}+p^\ell\]
 \[\leq\deg H_i=\u\]
 and consequently,
 \[\ddp\leq \frac{c!}{\q}\ord_{Q_T}\wF_{c-\q}-\sum_{i\not\in \TT} s_i\]
 \[\leq\frac{c!}{\q}\u-\sum_{i\not\in \TT} s_i\leq\ord I.\]
 But this contradicts our initial assumption that $\ddp >\ord I$. This gives (8) and ends the proof of the theorem.
 
\end{proof}

\bibliography{bibliography}

\providecommand{\bysame}{\leavevmode\hbox to3em{\hrulefill}\thinspace}
\providecommand{\MR}{\relax\ifhmode\unskip\space\fi MR }
\providecommand{\MRhref}[2]{%
  \href{http://www.ams.org/mathscinet-getitem?mr=#1}{#2}
}
\providecommand{\href}[2]{#2}
\begin{thebibliography}{Kaw07}

\bibitem[Abh56]{Abhyankar_56}
S.~Abhyankar, \emph{{Local uniformization on algebraic surfaces over ground
  fields of characteristic {$p\ne 0$}}}, Ann. of Math. (2) \textbf{63} (1956),
  491--526.

\bibitem[Abh66]{Abhyankar_3_Folds}
\bysame, \emph{{Resolution of singularities of embedded algebraic surfaces}},
  {Pure and Applied Mathematics, Vol. 24}, Academic Press, New York-London,
  1966.

\bibitem[BM97]{BM_Canonical_Desing}
E.~Bierstone and P.~Milman, \emph{{Canonical desingularization in
  characteristic zero by blowing up the maximum strata of a local invariant}},
  Invent. Math. \textbf{128} (1997), no.~2, 207--302.

\bibitem[BV13]{BV_Monoidal}
A.~Benito and O.~Villamayor, \emph{{Monoidal transforms and invariants of
  singularities in positive characteristic}}, Compos. Math. \textbf{149}
  (2013), no.~8, 1267--1311.

\bibitem[Cos87]{Cossart_Thesis}
V.~Cossart, \emph{{Poly{\`e}dre caract{\'e}ristique d'une singularit{\'e}}},
  Th{\`e}se d'{\'e}tat, Orsay, 1987.

\bibitem[Cos91]{Cossart_Contact_Maximal}
\bysame, \emph{{Contact maximal en caract{\'e}ristique positive et petite
  multiplicit{\'e}}}, Duke Math. J. \textbf{63} (1991), no.~1, 57--64.

\bibitem[CP08]{Cossart_Piltant_1}
V.~Cossart and O.~Piltant, \emph{{Resolution of singularities of threefolds in
  positive characteristic. {I}. {R}eduction to local uniformization on
  {A}rtin-{S}chreier and purely inseparable coverings}}, J. Algebra
  \textbf{320} (2008), no.~3, 1051--1082.

\bibitem[CP09]{Cossart_Piltant_2}
\bysame, \emph{{Resolution of singularities of threefolds in positive
  characteristic. {II}}}, J. Algebra \textbf{321} (2009), no.~7, 1836--1976.

\bibitem[Cut04]{Cutkosky_Book}
S.~D. Cutkosky, \emph{{Resolution of singularities}}, American Mathematical
  Society, Providence, R.I, 2004.

\bibitem[Cut09]{Cutkosky_3_Folds}
\bysame, \emph{{Resolution of singularities for 3-folds in positive
  characteristic}}, Amer. J. Math. \textbf{131} (2009), no.~1, 59--127.

\bibitem[Cut11]{Cutkosky_Skeleton}
\bysame, \emph{{A skeleton key to {A}bhyankar's proof of embedded resolution of
  characteristic p surfaces}}, Asian J. Math. \textbf{15} (2011), no.~3,
  369--416.

\bibitem[EH02]{EH}
S.~Encinas and H.~Hauser, \emph{{Strong resolution of singularities in
  characteristic zero}}, Comment. Math. Helv. \textbf{77} (2002), no.~4,
  821--845.

\bibitem[Gir75]{Giraud}
J.~Giraud, \emph{{Contact maximal en caract{\'e}ristique positive}}, Ann. Sci.
  {\'E}cole Norm. Sup. (4) \textbf{8} (1975), no.~2, 201--234.

\bibitem[Hau03]{Ha_BAMS_1}
H.~Hauser, \emph{{The {H}ironaka theorem on resolution of singularities (or:
  {A} proof we always wanted to understand)}}, Bull. Amer. Math. Soc. (N.S.)
  \textbf{40} (2003), no.~3, 323--403 (electronic).

\bibitem[Hau04]{Ha_Power_Series}
\bysame, \emph{{Three power series techniques}}, Proc. London Math. Soc. (3)
  \textbf{89} (2004), no.~1, 1--24.

\bibitem[Hau10]{Ha_BAMS_2}
\bysame, \emph{{On the problem of resolution of singularities in positive
  characteristic (or: a proof we are still waiting for)}}, Bull. Amer. Math.
  Soc. (N.S.) \textbf{47} (2010), no.~1, 1--30.

\bibitem[Hau14]{Ha_Obergurgl_2014}
\bysame, \emph{{Blowups and resolution}}, {The resolution of singular algebraic
  varieties}, Amer. Math. Soc., Providence, RI, 2014, pp.~1--80.

\bibitem[Hir64]{Hironaka_Annals}
H.~Hironaka, \emph{{Resolution of singularities of an algebraic variety over a
  field of characteristic zero. {I}, {II}}}, Ann. of Math. (2) 79 (1964),
  109--203; ibid. (2) \textbf{79} (1964), 205--326.

\bibitem[Hir84]{Hironaka_Bowdoin}
\bysame, \emph{{ Desingularization of excellent surfaces, Bowdoin 1967}},
  {Resolution of surface singularities, V. Cossart, J. Giraud, and U. Orbanz},
  {Lecture Notes in Mathematics}, vol. 1101, Springer-Verlag, 1984.

\bibitem[Hir12]{Hironaka_CMI}
\bysame, \emph{{Resolution of singularities}}, Manuscript distributed at the
  CMI Summer School 2012, 138 pp.

\bibitem[HP]{Indefinite_Increase}
H.~Hauser and S.~Perlega, \emph{{Cycles of Singularities appearing in the
  Resolution Problem in positive Characteristic}}, {J. Algebraic Geom.}, to
  appear.

\bibitem[HP16]{HP_Surfaces}
\bysame, \emph{{A new proof for the embedded resolution of surface
  singularities}}, Manuscript, 2016.

\bibitem[HW14]{HW}
H.~Hauser and D.~Wagner, \emph{{Alternative invariants for the embedded
  resolution of purely inseparable surface singularities}}, Enseign. Math.
  \textbf{60} (2014), no.~1-2, 177--224.

\bibitem[Kaw07]{Kawanoue_IF_1}
H.~Kawanoue, \emph{{Toward resolution of singularities over a field of positive
  characteristic. {I}. {F}oundation; the language of the idealistic
  filtration}}, Publ. Res. Inst. Math. Sci. \textbf{43} (2007), no.~3,
  819--909.

\bibitem[KM10]{Kawanoue_Matsuki}
H.~Kawanoue and K.~Matsuki, \emph{{Toward resolution of singularities over a
  field of positive characteristic (the idealistic filtration program) {P}art
  {II}. {B}asic invariants associated to the idealistic filtration and their
  properties}}, Publ. Res. Inst. Math. Sci. \textbf{46} (2010), no.~2,
  359--422.

\bibitem[Kol07]{Kollar_Book}
J.~Koll{\'a}r, \emph{{Lectures on resolution of singularities}}, Princeton
  University Press, Princeton, 2007.

\bibitem[Moh87]{Moh}
T.~T. Moh, \emph{{On a stability theorem for local uniformization in
  characteristic {$p$}}}, Publ. Res. Inst. Math. Sci. \textbf{23} (1987),
  no.~6, 965--973.

\bibitem[Moh96]{Moh_Newton_Polygon}
\bysame, \emph{{On a {N}ewton polygon approach to the uniformization of
  singularities of characteristic {$p$}}}, {Algebraic geometry and
  singularities ({L}a {R}{\'a}bida, 1991)}, {Progr. Math.}, vol. 134,
  Birkh{\"a}user, Basel, 1996, pp.~49--93.

\bibitem[Vil89]{Villamayor_Constructiveness}
O.~Villamayor, \emph{{Constructiveness of {H}ironaka's resolution}}, Ann. Sci.
  {\'E}cole Norm. Sup. (4) \textbf{22} (1989), no.~1, 1--32.

\bibitem[Vil92]{Villamayor_Patching}
\bysame, \emph{{Patching local uniformizations}}, Ann. Sci. {\'E}cole Norm.
  Sup. (4) \textbf{25} (1992), no.~6, 629--677.

\bibitem[Vil07]{Villamayor_Hypersurface}
\bysame, \emph{{Hypersurface singularities in positive characteristic}}, Adv.
  Math. \textbf{213} (2007), no.~2, 687--733.

\bibitem[W{\l}o05]{Wlodarczyk}
J.~W{\l}odarczyk, \emph{{Simple {H}ironaka resolution in characteristic zero}},
  J. Amer. Math. Soc. \textbf{18} (2005), no.~4, 779--822 (electronic).

\end{thebibliography}
\bibliographystyle{amsalpha}

\vfill\eject 

\end{document}